\documentclass[letterpaper,11pt]{article}
\pdfoutput=1 
\usepackage[utf8]{inputenc}
\usepackage{times,natbib} 
\usepackage{doi} 
\usepackage{url}
\usepackage{hyperref}
\hypersetup{colorlinks = true,citecolor=blue,urlcolor=cyan}
\usepackage{amssymb,amsmath,amsthm,amscd,cite,setspace}
\usepackage[margin=1.3in]{geometry}
\newcommand{\iid}{independent and identically distributed}
\newcommand{\V}{\text{\rm Var}\,}
\newcommand{\E}{\text{\rm E}\,}
\newcommand{\C}{\text{\rm Cov}\,}
\renewcommand{\leq}{\leqslant}
\renewcommand{\geq}{\geqslant}
\numberwithin{equation}{section}
\newtheorem{theorem}{Theorem}

\newtheorem{corollary}{Corollary}

\begin{document}

\title{Reverting Processes} 
\author
{{\sc \normalsize  By PETER CLIFFORD} \vspace{0.5ex} \\
{\em \normalsize Jesus College, Oxford, OX1 3DW, UK} \vspace{1.5ex} \\
{\sc \normalsize  and DAVID STIRZAKER} \vspace{0.5ex} \\
{\em \normalsize St.\ John's College, Oxford,  OX1 3JP, UK}}

\date{\vspace{-1ex}}

\maketitle
 
\begin{abstract}
We consider random processes that are history-dependent, in the sense that the distribution of the next step of the process at any time depends upon the entire past history of the process.  In general, therefore, the Markov property cannot hold, but it is shown that a suitable sub-class of such processes can be seen as directed Markov processes, subordinate to a random non-Markov directing  process whose properties we explore in detail. This enables us to describe the behaviour of the subordinated process of interest. Some examples, including reverting random walks and a reverting branching process, are given. \\


\end{abstract}

\section{Introduction}
History-dependent random processes arise naturally in seeking to describe many physical systems; a classic example is the self-avoiding random walk which seeks (among other applications) to capture aspects of the behaviour of certain long-chain molecules -- see \citet{HM,H}.  More recently, there has been considerable interest in other types of history-dependent random walks, such as: loop-erasing, Laplacian and self-reinforcing walks.  See \citet{P,T} and the references therein for details. 

Another classical class of history-dependent processes stems from the random (integer) sequences introduced by \citet{BSU}.  Among other things, they consider the sequence $(X_n; n\geqslant 1)$ defined by  
\begin{equation}\label{BSU}
X_{n+1} = X_n + X_{U(n)}, \ \ n \geqslant 2,
\end{equation}
where $X_1=x_1$ and $(U(n); n \geqslant 1)$ is a sequence of independent random variables such that for given $n$, $U(n)$ is uniformly distributed on $\{1,\dots,n\}$.

More recently, \citet{BK} have introduced two similar random integer sequences $(R_n; n \geqslant 1)$ and $(T_n; n \geqslant 1)$ defined as follows:
\begin{equation}\label{Rdef}
R_{n+1} = R_{U(n)} + X_n, \quad n \geqslant 1,
\end{equation}
with $R_1=0$ and
\begin{equation}\label{Tdef}
T_{n+1} = 1+ T_{U(n)}, \quad n \geqslant 1,
\end{equation}
with $T_1=0$, where $(U(n))$ is as in \eqref{BSU} and $(X_n; n \geqslant 1)$ is a sequence of independent identically distributed Rademacher random variables taking values $+1$ and $-1$ with probability $1/2$.  We shall refer to the sequence $(R_n)$ defined in \eqref{Rdef} as the simple symmetric reverting random walk (avoiding the label `random random walk' which was originally applied in \citet{BK} because this term is generally used in the context of an entirely different process). \citeauthor{BK} obtain the p.g.f.\ for both $R_n$ and $T_n$, together with their moments and consider other aspects of the behaviour of these processes.

Here, we initially investigate the properties of the stochastic processes $(R_n)$ and $(T_n)$; first establishing some limit theorems for $T_n$, as $n \to \infty$. Then we exploit a connexion between the two processes to obtain a structural understanding and a limit theorem for $R_n$.
These methods are then applied to several natural extensions, variants and generalisations of the reverting random walk
\section{Reverting random walks}\label{sec:2}
We define the general uniformly reverting random walk by the recurrence
\begin{equation}\label{GenRdef}
R_{n+1} = R_{U(n)} + X_n, \quad n \geqslant 1,
\end{equation}
where the sequence $(U_n)$ is as in \eqref{BSU}, but now $X_1,X_2,\dots$ are independent and identically distributed with an arbitrary distribution. Without appreciable loss of generality, we shall assume that $R_1=0$.

First we note that there is a useful connexion between the sequence $(T_n)$ defined in \eqref{Tdef} and the sequence $(R_n)$ above, since as we will show
\begin{equation}\label{Rstar}
R_n = \sum_{r=1}^{T_n} X^*_r, \quad n \geqslant 1,
\end{equation}
where the empty sum is $0$, and the sequence $(X^*_r; r \geq 1)$, comprises \iid\ random variables having the same distribution as $X_1$.  That is to say, $(R_n)$ is a subordinate random walk directed by the process $(T_n)$, which we refer to as the \emph{reverting clock}. 

In fact, by coupling the recursions of \eqref{Tdef} and \eqref{GenRdef} together, $(X^*_r; 1 \leq r \leq T_n)$  can be shown to be a random subsequence of $(X_r; 1 \leq r < n)$, selected by the shared reversions $(U(r),1 \leq r < n)$.  This is most simply demonstrated by induction.  

Suppose, for the purposes of induction, that it is known that for all $k$ in $\{2,\dots,n\}$ 
the value $R_k$ is given by 
$$R_k = \sum_{j=1}^{T_k} X_{i_j(k)},$$
where $\{i_j(k); j=1,\dots,T_k\}$ is a subsequence of $1,\dots,(k-1)$.  The base case is valid, since $R_2=X_1$. Now consider $R_{n+1}$. Suppose that $U(n) = k$ then $R_{n+1}=R_k +X_n$ and $T_{n+1} = 1+ T_k$.  It follows that 
$$R_{n+1} = \sum_{j=1}^{T_k} X_{i_j(k)} + X_n = \sum_{j=1}^{T_{n+1}} X_{i_j(n+1)},$$
where
$$  i_j(n+1) = \begin{cases}
				i_j(k) & \text{if $j<T_{n+1}$},\\
				n & \text{if $j=T_{n+1}$},
				  		 \end{cases}
$$
which completes the inductive proof. 

Since the processes $(U(n))$ and $(X_n)$ are independent, it follows that $R_n$ has the structure of \eqref{Rstar}. In other words $R_n$ is the value of a random walk observed at a random time $T_n$.  We therefore turn to an analysis of the reverting clock, $(T_n)$.

Writing $m_n=\E T_n$ and $v_n = \V T_n$ we have, by conditioning on the $n$th reversion $U(n)$,  
\begin{equation}\label{BK1}
m_{n+1} = 1 + n^{-1}\sum_{k=1}^n m_k, \quad n \geq 1,
\end{equation}
and hence it can be shown that
\begin{equation}\label{BK2}
m_n = \sum_{k=1}^{n-1} k^{-1} = \log n + \gamma + n^{-1}/2 + o(n^{-1}),
\end{equation}
as $n \to \infty$, where $\gamma$ is Euler's constant.  Likewise
\begin{equation}\label{BK3}
v_n = \sum_{k=1}^{n-1} (k^{-1}-k^{-2}) = \log n + \gamma - \pi^2/6 +3 n^{-1}/2 +  o(n^{-1}),
\end{equation}
as $n \to \infty$.  In addition, the probability mass function (p.m.f.) of $T_n$ satisfies the recurrence
\begin{equation}\label{BK4}
\Pr(T_{n+1}= x) = n^{-1} \sum_{k=1}^n \Pr (T_k = x-1).
\end{equation}
The results \eqref{BK1}--\eqref{BK4} were obtained in \citet{BK}; and \eqref{BK4} was used to yield the probability generating function (p.g.f.) of $T_n$.  

Furthermore, by considering a continuized version of the difference equation \eqref{BK4}, it is argued in \citet{BK} that the distribution of $T_n$, suitably scaled should approach a normal limit as $n \to\infty$.  Here, we show that this result can be obtained by applying a central limit theorem to $T_n$ directly, as follows.
\begin{theorem} \label{thm:1}
For each fixed $x$, as $n \to \infty$, 
$$\Pr\left( \frac{T_n - m_n}{v_n^{1/2}} \leq x \right) \to \Phi(x),$$
where $\Phi$ is the standard normal distribution function.
\end{theorem}
\begin{proof}
Let $T_n$ have p.g.f.\ $G_n(s) = \E s^{T_n}$, then by conditioning on the reversion $U(n)$ we have
$$G_{n+1}(s) = \frac{s}{n} \sum_{k=1}^n G_k(s),
$$
whence, easily,
\begin{eqnarray}\label{Gprod}
G_{n+1}(s) 
&=& \frac{1}{n!} s(1+s)\dots(n-1+s) = s\left(\frac12 + \frac{s}2\right) \dots\left(\frac{n-1}n + \frac{s}n\right)\nonumber \\
&=& \prod_{k=1}^n g_k(s),
\end{eqnarray}
where $g_k(s)$ is the p.g.f\ of a random variable $Z_k$ taking the values $0$ and $1$ with probabilities $(k-1)/k$ and $1/k$, respectively.  In other words, the distribution of $T_{n+1}$ is that of the sum of $n$ independent Bernoulli variables with varying success probabilities. 

Writing $T_{n+1} = \sum_{k=1}^n Z_k$, we have $\rho_{n+1}=  \sum_{k=1}^n |Z_k - \E Z_k|^3 = \log n + o(\log n)$, and hence $\rho_n/v_n^{3/2} \to 0$ as $n \to \infty$. It follows that 
$$\Pr \left(\frac{T_n - m_n}{v_n^{1/2}} \leq x\right) \to \Phi(x), \quad \text{as $n \to \infty$},$$ from a central limit theorem for non-identically distributed independent summands; see for example \citet[p.\ 287]{Loeve} . 
\end{proof}
Equivalently, from \eqref{BK2} and \eqref{BK3}  
\begin{equation} \label{slutsky}
\frac{T_n - m_n}{v_n^{1/2}} = \frac{T_n -\log n}{(\log n)^{1/2}} A_n + B_n,
\end{equation}
where $A_n \to 1$ and $B_n \to 0$ as $n \to \infty$, and using Slutsky's lemmas we have
$$
\Pr \left(\frac{T_n -\log n}{(\log n)^{1/2}} \leq x\right) \to \Phi(x), \quad \text{as $n \to \infty$}.
$$ 
For later reference, it is useful to explain why $T_{n+1}$ can be expressed as $\sum_{k=1}^n Z_k$, the sum of $n$ independent Bernoulli variables.  Let $W_1 = \max\{m: \sum_{k=m}^n  Z_k =1\}$, i.e.\ the index of the last success in the sequence of Bernoulli trials.
It is straightforward to show that $W_1$ is uniformly distributed on $\{1,\dots,n\}$ (for example by considering the sequence in reverse order).  In the defining relation, $T_{n+1} = 1 + T_{U(n)}$, of \eqref{Tdef} we can then take $U(n)$ to be $W_1$.  To determine the value of the process at $W_1$ we must revert to an earlier time $W_2$ uniformly distributed on $\{1,\dots,W_1-1\}$, provided of course that  $W_1 > 1$. The value of $T_{n+1}$ is then 2 more than the value at time $W_2$.  Conditional on $W_1=w$, the variable $W_2$ can be represented as $\max\{m: \sum_{k=m}^{w-1}  Z_k =1\}$.  Unconditionally it then has the representation $W_2 = \max\{m: \sum_{k=m}^n  Z_k =2\}$.  The reversions can then be continued with  $W_j = \max\{m: \sum_{k=m}^n  Z_k =j\}$ at the $j$th stage, until $W_{j} = 1$ for the first time at some value $j=\tau$.  Since the count is incremented by one at each such reversion, it follows that $T_{n+1} = \tau$.  We note that the indices $\{k:Z_k=1\}$ mark the time-points in the history of $T_{n+1}$ at which  
increments occur.  It follows that the value of $T_{n+1}$ is the number of these time-points, i.e.\ the sum of the associated Bernoulli variables.     
\subsection{The time integral}
Next we consider the time-integral of the process $(T_n)$, namely the process $(S_n)$ where $S_n = \sum_{k=1}^n T_k$, with $S_1=0$.  This exhibits different convergent behaviour.  Let
$$M_n = n^{-1}(S_n - \E S_n),$$
so that $M_1 = 0$. Then we have this
\begin{theorem} \label{thm:2}
$(M_n)$ is an $L_2$-bounded martingale with respect to the natural filtration $(\mathcal{F}_n)$ generated by $(T_n)$ and $(U(n))$. Consequently, as $n \to \infty$, $M_n$ converges almost surely and in mean-square to a non-degenerate random variable $M$.  Furthermore, $\E M = 0$ and $\V M = 2- \pi^2/6.$
\end{theorem}
\begin{proof}
First note that $\E (T_{n+1}|\mathcal{F}_n ) = 1 + n^{-1} S_n$ and thus $\E T_{n+1} = 1 + n^{-1} \E S_n$.  Hence $n^{-1}\E S_n = \sum_{k=2}^n k^{-1}$ from \eqref{BK2}. Since $S_n$ is $\mathcal{F}_n$-measurable,
\begin{eqnarray*}
 \E(M_{n+1}|\mathcal{F}_n) &=&  \frac{1}{n+1}[\E(T_{n+1} | \mathcal{F}_n) + S_n]- \E \left(\frac{S_{n+1}}{n+1}\right)\\
                          &=&  \frac{1}{n+1}[1+n^{-1}S_n + S_n] - \sum_{k=2}^{n+1} k^{-1}\\
                          &=&  n^{-1} S_n - \sum_{k=2}^{n} k^{-1} = n^{-1}(S_n - \E S_n) = M_n,
\end{eqnarray*}                          
which is the martingale condition for $(M_n)$. Since $M_1 =0$, it follows that $\E(M_n)=0$, for all $n$. 

We now show that $\E(M_n^2)$, or equivalently $\V(S_n/n)$, is less than $K$ for all $n$, for some $K < \infty$, which will establish the desired convergence by the martingale mean-square convergence theorem \citep{Doob}.  Conditioning on $\mathcal{F}_{n-1}$ we have
\begin{equation}\label{VarSn}
\V S_n = \E [\V (S_n|\mathcal{F}_{n-1})] + \V[\E(S_n|\mathcal{F}_{n-1})].
\end{equation}  
Since $\E (S_n|\mathcal{F}_{n-1}) = S_{n-1} + 1 + S_{n-1}/(n-1) = 1+ [n/(n-1)] S_{n-1}$, we thus have
\begin{equation} \label{VES}
\V[\E(S_n|\mathcal{F}_{n-1})] = \left(\frac{n}{n-1}\right)^2 \V S_{n-1}.
\end{equation}
Furthermore, $\V (S_n|\mathcal{F}_{n-1}) = \V (S_{n-1}+T_n|\mathcal{F}_{n-1}) = \V(T_n |\mathcal{F}_{n-1})$ and since  
\begin{eqnarray*} 
\V T_n 
&=& \E [\V (T_n|\mathcal{F}_{n-1})] + \V[\E(T_n|\mathcal{F}_{n-1})] \\
&=& \E [\V (T_n|\mathcal{F}_{n-1})] + \V[1+(n-1)^{-1}S_{n-1}]\\
&=& \E [\V (T_n|\mathcal{F}_{n-1})] + \frac{1}{(n-1)^2}\V S_{n-1},
\end{eqnarray*} 
we have,
\begin{equation} \label{EVS}
\E [\V (T_n|\mathcal{F}_{n-1})] = \E [\V (S_n|\mathcal{F}_{n-1})] = \V T_n - \frac{1}{(n-1)^2} \V S_{n-1}.
\end{equation}
Substituting \eqref{VES} and \eqref{EVS} in \eqref{VarSn}, yields
$\V S_n = \V T_n + (n+1)/(n-1) \V S_{n-1},$
or  equivalently
$$Q_n = Q_{n-1} + \frac{v_n}{n(n+1)},$$
where $Q_n= n/(n+1)\V (S_n/n)$ and $v_n = \V T_n$.  It follows that
$$ Q_n = \sum_{k=1}^n \frac{v_k}{k(k+1)}.$$
To establish that $\V(S_n/n) < K$ for all $n$, it is only remains to show that $Q_n$ has a finite limit as $n \to \infty$.  Interchanging the order of summation (justified since all terms are non-negative) we have 
\begin{eqnarray*}
\sum_{k=1}^\infty \frac{v_k}{k(k+1)} &=& \sum_{k=2}^\infty \frac{1}{k(k+1)}\sum_{j=1}^{k-1}\left(\frac{1}{j} - \frac{1}{j^2}\right)\\
&=& \sum_{j=1}^\infty \left(\frac{1}{j} - \frac{1}{j^2}\right) \sum_{k=j+1}^\infty \left(\frac{1}k - \frac{1}{k+1}\right) =  
\sum_{j=1}^\infty \left(\frac{j-1}{j^2}\right) \frac{1}{j+1}\\
&=&\sum_{j=1}^\infty \left(\frac{2}{j(j+1)}-\frac{1}{j^2}\right) = 2 - \frac{\pi^2}{6}.
\end{eqnarray*} 
Furthermore since $\lim_{n \to \infty} \E (M_n^2) = \lim_{n \to \infty} Q_n$, and $M_n$ converges in mean-square to $M$, we have shown that $\E (M^2) = 2 - \pi^2/6.$
\end{proof}   
The determination of the distribution of M remains an open problem.   
Computer simulations indicate that $M$ has a continuous distribution with an asymmetric density that is unimodal and skewed to the right.   

However, we note that $T_n \leq n$ and $S_n \leq n(n+1)/2$, so the martingale differences satisfy
$$|M_{n+1} - M_n| = \left|\frac{T_{n+1} - m_{n+1}}{n+1} - \frac{S_n- \E S_n}{n(n+1)} \right| \leq \frac32.$$
We can thus put strong bounds on large deviations of $M_n$, by Hoeffding's inequality.
\begin{corollary}\label{cor:1}
For $m,n >0$, 
$$\C(T_n, T_{n+m}) = \frac{n-1}{n} \V M_{n-1} + \frac1{n}\V T_n.$$
\end{corollary}
\begin{proof}
Let $(\mathcal{F}_n)$ be the filtration as in theorem \ref{thm:2}, then
$$\E(T_{m+n} | \mathcal{F}_{n+m-1}) = 1 + \frac{S_{n+m-1}}{n+m-1} = 1 + M_{n+m-1} + \E M_{n+m-1},$$
and using the martingale property of $(M_n)$ we have
$$\E(T_{m+n} | \mathcal{F}_{n}) = 1 + M_{n} + \E M_{n+m-1}.$$
Using the conditional covariance relationship, we have
\begin{eqnarray*}
\C(T_n, T_{n+m}) &=& \E[\C(T_n,T_{n+m}| \mathcal{F}_n)] + \C[T_n, \E(T_{m+n}|\mathcal{F}_n)]\\
 &=& 0+ \C(T_n, M_n),
\end{eqnarray*}
and
\begin{eqnarray*}
\C(T_n, M_n) 
&=& \E[\C(T_n,M_n|\mathcal{F}_{n-1})] + \C[\E(T_n|\mathcal{F}_{n-1}),\E(M_n|\mathcal{F}_{n-1})]\\
&=& \E[\C(T_n,(T_n+S_{n-1})/n|\mathcal{F}_{n-1})] + \C(1+M_{n-1} + \E M_{n-1},M_{n-1})\\
&=& \E[\V(T_n | \mathcal{F}_{n-1})]/n + \V M_{n-1}\\
&=& \frac{n-1}{n} \V M_{n-1} + \frac1{n}\V T_n, 
\end{eqnarray*} 
using \eqref{EVS} at the final stage. 
\end{proof}  
We return to the reverting random walk.  It is immediate from \eqref{Rstar} and \eqref{Gprod} that the characteristic function of $R_{n+1}$ is
given by
$$\Psi_{n+1}(\theta) = G_{n+1}(\phi_X(\theta)) = \prod_{k=1}^n \psi_k(\theta),$$
where $\psi_k(\theta) = (k-1)/k + \phi_X(\theta)/k$ is the characteristic function of $Z_k X_1$.

For the simple reverting random walk with 
$\Pr(X_1=1)=p$ and $\Pr(X_1 =-1) = q$, where $q=1-p$, the p.g.f.\ of $R_{n+1}$ is given by 
$$ \E\!\!\left(s^{R_{n+1}}\right) =  \prod_{k=1}^n \left(\frac{p s}{k} + \frac{k-1}{k} + \frac{qs^{-1}}{k}\right) 
= \sum_{k=1}^n \genfrac{[}{]}{0pt}{}{n}{k}\left(p s + qs^{-1}\right)^k,$$
where $\genfrac{[}{]}{0pt}{}{n}{k}$ are the Stirling numbers of the first kind.
From the first form we see that $R_{n+1}$ has the same distribution as the position of an inhomogeneous random walk after $n$  independent steps of size $1, 0$ and $-1$ with corresponding  probabilities $p/k$, $(k-1)/k$ and $q/k$ for the $k$th step, $k=1,2,\dots,n$. From the second form, we can obtain the probability mass function of $R_{n+1}$ in terms of Stirling numbers and binomial coefficients. We note that this supplies the solution to an open problem posed in \citet{BK}.

Finally, we note that if the sequence $(\sum_{k=1}^n X_k;n \geq 1)$ obeys a central limit theorem as $n \to \infty$ then so does the sequence    
$(R_n)$.  This follows from Anscombe's theorem \citep{A}, when we remark that $T_n /\log(n) \overset{\text{\rm P}}{\to} 1$. \citet{BK} consider a continuized limit of the difference equation satisfied by the p.m.f.\ of the simple reverting random walk in the case $p=1/2$, and obtain a diffusion equation from which they postulate the limiting normal distribution.  Their result is a corollary of the general observation above.

\section{Variants of reverting random walks}
There are many obvious natural variants, extensions and generalizations of the reverting random walk; we briefly consider some of them.
\subsection{A special class of non-uniform reversions}
Suppose now that the reverting variable $U(n)$ has a non-uniform distribution on $\{1,\dots,n\}$.  A flexible class of such distributions can be constructed by considering a sequence of positive numbers, $\alpha_1,\alpha_2,\dots$, and defining $U(n)$ by
$$\Pr(U(n)=k) = \frac{\alpha_k}{\sum_{j=1}^n \alpha_j}, \quad n \geq 1.$$  As before, the reverting clock is defined by $T_{n+1} = 1 + T_{U(n)}$, with $T_1=0$ and the reverting walk is defined by $R_{n+1} = R_{U(n)} + X_n$ with $R_1 =0$.

Proceeding as in the uniform case, we have 
$$G_{n+1}(s) =  \E\left(s^{T_{n+1}}\right) = \frac{s}{\sum_{j=1}^n \alpha_j}\sum_{k=1}^n \alpha_k G_k(s),$$
so that 
$$G_{n+1}(s) = \left(\frac{s \alpha_n}{\sum_{j=1}^n \alpha_j} + \frac{\sum_{j=1}^{n-1} \alpha_j}{\sum_{j=1}^n \alpha_j}\right)G_n(s)= \prod_{k=1}^n(s p_k + q_k),$$
where $p_k = \alpha_k/\sum_{j=1}^k \alpha_j$, and $p_k+q_k =1$.  

We see that $T_{n+1}$ can be represented as $T_{n+1} = \sum_{k=1}^n Z_k$ where $Z_1,\dots,Z_n$ are independent Bernoulli variables with $\Pr(Z_k=1) = p_k$. Consequently 
$$m_{n} = \E T_{n} = \sum_{k=1}^{n-1} p_k\quad \text{and} \quad v_{n} = \V T_{n} =  \sum_{k=1}^{n-1} p_k q_k.$$
First note that if $v_n$ has a finite limit, $v$ then $T_n - m_n$ converges in distribution to a random variable with variance $v$; and in general, if $v_n/m_n^2 \to 0$ as $n \to \infty$ then $T_n/m_n \overset{\text{\rm P}}{\to} 1$.  

As in the case of uniform reversion, a central limit theorem applies to $T_n$ when $\rho_n/v^{3/2}_n \to 0$ as $n \to \infty$, where
$\rho_{n} = \sum_{k=1}^{n-1} \E|Z_k - p_k|^3$. For illustration, we consider reversion distributions for which $\alpha_k = k^\beta$, for arbitrary  $\beta$. For this class, the asymptotic behaviour of $p_k$ as $k \to \infty$ is given by
$$
p_k \sim  
\begin{cases}
	(\beta+1) k^{-1} & \text{if $\beta > -1$}\\
	(k \log k)^{-1} & \text{if $\beta = -1$} \\
	C_\beta k^{\beta} & \text{if $\beta < -1$.}
\end{cases}
$$
It follows that when $\beta > -1$, the quantities $m_n$,$v_n$ and $\rho_n$ each grow as $\log n$; for $\beta = -1$, the growth is  $\log(\log n)$ and when $\beta < -1$ each has a finite limit. We conclude that $\rho_n/v^{3/2}_n \to 0$ when $\beta \geq -1$ and hence, in this case, the distribution of $(T_n - m_n)v_n^{-1/2}$ converges to the standard normal form.  We also conclude that $T_n/m_n \overset{\text{\rm P}}{\to} 1$ under this condition, so that by Anscombe's theorem \citep{A}, the associated reverting random walk obeys the central limit theorem, provided the variables $(X_n)$ meet the standard summand conditions of this theorem.

Finally, we note that a martingale can be exploited in a fashion similar to that of theorem \ref{thm:2} and corollary \ref{cor:1}. 
\begin{theorem}\label{thm:x}
Let 
$$M_{n} = \frac{\sum_{k=1}^n \alpha_k (T_k - m_k)}{\sum_{k=1}^n \alpha_k}, \quad n \geq 1.
$$
If $\sum_{k=1}^\infty p^2_k v_k < \infty$ then $(M_n)$ is an $L_2$-bounded martingale with respect to the natural filtration $(\mathcal{F}_n)$ generated by $(T_n)$ and $(U(n))$. Consequently, as $n \to \infty$, $M_n$ converges almost surely and in mean-square to a non-degenerate random variable $M$, with mean zero and finite variance.    
\end{theorem}
\begin{proof}
The proof of the martingale property is similar to that in theorem \ref{thm:2}, where $\alpha_k=1$ for all $k$. It is straightforward to check that $\E M_n =0$ for all $n$. The proof in theorem \ref{thm:2} can also be modified to show that $\V (M_n) < K$ for all $n$ under the conditions of the theorem.   

Defining $S_n = \sum_{k=1}^n \alpha_k T_k$, we have 
$\E(S_n | \mathcal{F}_{n-1}) = \alpha_n + S_{n-1} A_n/A_{n-1},$
where $A_n = \sum_{k=1}^n \alpha_k$, so that
\begin{equation}\label{VE1}\V[\E(S_n|\mathcal{F}_{n-1})] =  \V (S_{n-1}) A_n^2/A_{n-1}^2.\end{equation}
As in theorem \ref{thm:2}, 
$ \V(S_n | \mathcal{F}_{n-1}) = \alpha_n^2 \V (T_n | \mathcal{F}_{n-1}),$
and consequently
\begin{equation}\label{EV1}
\E [\V (S_n|\mathcal{F}_{n-1})] = \alpha_n^2\left\{ v_n - A_{n-1}^{-2} \V(S_{n-1})\right\},
\end{equation}
where $v_n = \V(T_n)$. From \eqref{VE1} and \eqref{EV1} we then have
$$\V(S_n) = \alpha_n^2 v_n + \V(S_{n-1})[A_{n}^2 - \alpha_n^2]/A_{n-1}^{2}$$
and dividing by $A_n^2$
\begin{equation}\label{V1}
\V(M_n) = p_n^2 v_n + (1-p_n^2) \V(M_{n-1}),
\end{equation}
where $p_n = \alpha_n/\sum_{j=1}^n \alpha_j$.  

Define $J_n = [\prod_{k=2}^n (1-p_k^2)]^{-1}$.  Since $\sum_{k=1}^\infty p^2_k v_k < \infty$ and $v_k$ is increasing, it follows that $\sum_{k=1}^\infty p^2_k < \infty$ and hence $J_n$ has a non-zero finite limit as $n \to \infty$.  Now define $Q_n$ to be $J_n \V (M_n)$. Then from \eqref{V1} we have $Q_n = p_n^2 v_n J_n +Q_{n-1},$
so that 
$$Q_n = \sum_{k=2}^n p_k^2 J_k v_k, \quad n \geq 2.$$  It follows that $Q_n$ and hence $\V(M_{n})$ has a finite limit as $n \to \infty$ and we conclude that $\V (M_n) < K$ for some fixed $K$ for all $n$. 
\end{proof} 
\subsection{Occasionally reverting walks}
In general, one might permit the walk to revert only at the instants of some suitable point process on the positive integers.  We consider a simple case of this, in which the walk attempts to revert with probability $q>0$ (independently of the past), or continues without reversion with probability $p$, where $p+q=1$.  As in section \ref{sec:2}, the walk at any step $n$ can be represented in terms of a reverting clock, $(T_n)$, given by the recursion
\begin{equation}\label{ORWdef}
T_{n+1} = 1 + T_{V(n)}, \quad n \geq 1,\quad \text{with} \quad V(n) = n (1-\mathcal{I}_n) + U(n)\mathcal{I}_n,
\end{equation}
where $(\mathcal{I}_n)$ are independent Bernoulli variables each with success probability $q$ and as before the reversions $(U(n))$ are independent with $U(n)$ uniformly distributed on $\{1,\dots,n\}$.    
Thus the clock advances one unit of time with probability $p+n^{-1}q$, stays still with probability $n^{-1}q$ and otherwise jumps backwards.

We can investigate this generalized reverting clock through its bivariate generating function:
$$G(s,z) = \sum_{k=1}^\infty z^{k-1} \E\!\!\left(s^{T_k}\right) = \sum_{k=1}^\infty z^{k-1}  G_k(s),$$
in an obvious notation.  Note that $G_1(s) =1$ and $G_2(s) =s$.  

By conditioning on the events at stage $n$ we have
$$G_{n+1}(s) = s p G_{n}(s) + \frac{s q}{n} \sum_{k=1}^n G_k(s), \quad n \geq 1.$$
Rearranging, this gives
$$(n+1)h_{n+2}(s) = n s p h_{n+1}(s) + (s-1) G_{n+1}(s),$$
where $h_n(s) = G_n(s)- G_{n-1}(s).$  Defining $H(s,z)= \sum_{n=1}^\infty z^{n-1} h_{n}(s)$, multiplying both sides by $z^n$ and summing over all $n$, gives
$$\frac{\partial H(s,z)}{\partial z} = z s p\frac{\partial H(s,z)}{\partial z} + \frac{(s-1)}{1-z} H(s,z), $$
and hence integrating and using $H(s,z) = (1-z)G(s,z),$ 
\begin{equation} \label{G}
G(s,z) = (1 - s p z)^{(s-1)/(1- s p)}(1-z)^{-q s/(1- s p)}.
\end{equation}
Recalling that 
$$\frac{\partial \log G(s,z)}{\partial s}\bigg|_{s=1} = (1-z) \sum_{k=1}^\infty z^{k-1} m_k,$$
where $m_k = \E T_k$, we compute
$$(1-z) \sum_{k=1}^\infty z^{k-1} m_k = q^{-1} [\log(1- z p) - \log(1-z)].$$
Thus
\begin{equation} \label{T}
m_{n+1} = q^{-1} \left\{\sum_{k=1}^n k^{-1}(1 - p^k)\right\} = q^{-1}\{ \log n + \gamma + \log q\} + O(n^{-1}),
\end{equation}
as $n \to \infty.$  Effectively this reflects the fact that times between reversions are now geometrically distributed random variables with mean $1/q$; compare with the mean of the basic reverting clock in \eqref{BK2}.  Higher moments of $T_n$ and the p.g.f.\ $G_n(s)$ are likewise available, in principle, from $G(s,z)$.  

The implications for $(R_n)$, the occasionally reverting random walk subordinate to the process $(T_n)$, are essentially the same as in section \ref{sec:2}.  In any particular case, the bivariate characteristic function of the walk may be written as 
$$\phi(\theta,z) = \sum_{k=1}^\infty z^{k-1} \E({\rm e}^{i \theta R_k}) = G({\rm e}^{i\theta},z),$$
where $G(\cdot,\cdot)$ is given by  \eqref{G}.  

If the variables $(X_n)$ are such that their partial sums obey the central limit theorem then, suitably standardized, $R_n$ has a normal limiting distribution.  Again this follows from Anscombe's theorem, for example by showing that $m_n^{-2}\V( T_n) \to 0$ as $n \to \infty$, and hence that $m_n^{-1} T_n \overset{\text{\rm P}}{\to} 1$.  Note that $m_n^{-2}\V( T_n) \to 0$ if and only if $m_n^{-2}w_n \to 1$, where $w_n = \E(T_n^2)$.  We establish this by deriving an expression for $w_n$ directly as follows.  

Conditioning on the events at stage $n$, we have 
\begin{equation} \label{m}
m_{n+1} = 1 + p m_n +\frac{q}{n} \sum_{k=1}^n m_k
\end{equation}
and 
\begin{equation} \label{w}
w_{n+1} = 1+2pm_n + p w_n +\frac{2q}{n} \sum_{k=1}^n m_k + \frac{q}{n} \sum_{k=1}^n w_k.
\end{equation}
Differencing \eqref{w} and using \eqref{m}, we have
$$\frac{n+1}{p^n} (w_{n+2}-w_{n+1}) = \frac{n}{p^{n-1}} (w_{n+1} - w_n) + \frac{2(1-p^{n+1})}{q p^{n}} + 
\frac{(2 m_{n+1} -1)}{p^{n}},$$
where $w_1 = m_1 = 0$, and $w_2 = m_2 = 1$.
Thus,
\begin{eqnarray*}
w_{n+1}-w_n 
&=& \frac{p^n}{n+1} \sum_{k=1}^{n+1} \left[\frac{2(1-p^k)}{q p^{k-1}} + \frac{(2 m_k -1)}{p^{k-1}}\right]\\
&=& \frac{(1+p)(1-p^{n+1})}{q^2(n+1)} - \frac{2p^{n+1}}{q} + \frac{p^n}{n+1} \sum_{k=1}^{n+1} \frac{2 m_k}{p^{k-1}}.
\end{eqnarray*}
For large $n$, using \eqref{T}, we then have $w_{n+1}-w_n  = 2(q^2n)^{-1} \log n + o(n^{-1} \log n)$. It follows that 
$$ w_n = q^{-2} \log^2 n + o(\log^2 n),$$
and hence, again with \eqref{T}, we have  $m_n^{-2} w_n \to 1$ as $n \to \infty$, as claimed.

We will show that, suitably standardized, $T_n$ has a limiting normal distribution. 
\begin{theorem} \label{thm:1.1}
Let $(T_n)$ be the process defined in \eqref{ORWdef}, and let $m_n$ and $v_n$ be the mean and variance of $T_n$. 
For each fixed $x$, as $n \to \infty$, 
$$\Pr\left( \frac{T_n - m_n}{v_n^{1/2}} \leq x \right) \to \Phi(x),$$
where $\Phi$ is the standard normal distribution function.
\end{theorem}
\begin{proof}
As in theorem \ref{thm:1} and the discussion thereafter, we will show that $T_{n+1}$ can be represented as a sum of Bernoulli variables: $\sum_{k=1}^n Z_k$, where $\{k:Z_k=1\}$ mark the time-points in the history of $T_{n+1}$ at which its increments occur.  The complication here is that the variables are no longer independent; they form an inhomogeneous first-order Markov chain.  

The chain is most conveniently defined in reverse time.  We start by defining $Z_n$ to be $1$ when $V(n)=n$ and zero otherwise, so that  $\Pr(Z_n=1) = p+q/n$, from \eqref{ORWdef}.  Note that the increment  $T_{n+1} = 1 + T_{n}$ occurs if and only if $Z_n=1$. If $Z_n=1$, we define $Z_{n-1}$ to be $1$  when $V(n-1) = n-1$ and $0$ otherwise. The event $Z_{n-1}=1$ corresponds to the increment  $T_n= 1 + T_{n-1}$ and the conditional probability of this event given $Z_n=1$ is then $p + q/(n-1)$.

Now suppose that $Z_n=0$. This means that $V(n) <n$. In this case, the conditional probability that $V(n) = n-1$ is $1/(n-1)$, resulting in the increment $T_{n+1} = 1 + T_{n-1}$ at the time-point $n-1$.   The transition probabilities are  
\begin{eqnarray*}
\Pr(Z_{n-1} = 1| Z_n=1) &=& p+q/(n-1) \\
\Pr(Z_{n-1} = 1| Z_n=0) &=& 1/(n-1)
\end{eqnarray*}  
More generally at the $k$th step backwards, we have 
\begin{equation}\label{pi}
\pi_{11}^{(k)} = p+\frac{q}{n-k}, \quad \pi_{01}^{(k)} = \frac{1}{n-k}, \quad \pi_{10}^{(k)} = \frac{q(n-k-1)}{n-k}, \quad \pi_{00}^{(k)} = \frac{n-k-1}{n-k},
\end{equation}
where $\pi_{ij}^{(k)} = \Pr(Z_{n-k} = j| Z_{n-k+1}=i)$.  

To apply Dobrushin's central limit theorem for inhomogeneous Markov chains \citep{D,SV} we must show that $$\alpha_n^{3} \sum_{k=1}^n \V(Z_k) \to \infty$$ as $n \to \infty$, where
$$\alpha_n = \min_{1 \leq k < n} \big[1 -\max_{i,j,\ell} \big|\pi_{i\ell}^{(k)}-\pi_{j\ell}^{(k)}\big|\big].$$   
Direct calculation from \eqref{pi} gives $\alpha_n \geq q > 0$. 

The variance of $Z_k$ is $p_k(1-p_k)$ where $p_k = \Pr(Z_k=1)$, with bounds 
$$ \min\{\pi_{01}^{(n-k)}, \pi_{11}^{(n-k)}\} \leq p_{k} \leq  \max\{\pi_{01}^{(n-k)}, \pi_{11}^{(n-k)}\}, \quad 1 \leq k <n.$$
From \eqref{pi} we then have 
$$p_k(1-p_k) \geq \frac{1}{n-k} \left[\frac{q(n-k-1}{n-k}\right], \quad 1 \leq k <n,$$
and hence $\sum_{k=1}^{n-1} \V(Z_k) \to \infty$ as $n \to \infty$ as required.
\end{proof}

Finally we note the existence of a martingale related to the integrated process $(S_n)$ where $S_n = \sum_{k=1}^n T_k$, as follows.  
Let  $N(n) = \max\left\{k: \sum_{i=1}^k \mathcal{I}_i =n\right\}$ and 
$$M_n = S_{N(n)}/N(n) - \sum_{k=1}^{n-1} \sum_{r=1}^{\infty} \frac{q p^{r-1} r(r+1)}{2(N(k)+r)}.$$
\begin{theorem} $(M_n)$ is a martingale with respect to the filtration $(\mathcal{F}^*_n)$ where $\mathcal{F}^*_n$ is the $\sigma$-field generated by $\{(V(k),\mathcal{I}_k);k \leq N(n)\}$.   
\end{theorem}
\begin{proof}
The time intervals between reversions are independent geometric random variables, which we denote by $(Y_n;n\geq 1)$; thus $ \sum_{k=1}^n Y_k = N(n) +1.$  Note that $T_{N(n)+k} = T_{N(n)+1} + k -1,$ for $k=1,\dots,Y_{n+1}$ since this corresponds to a period in which there are no reversions.  It follows that 
$$S_{N(n+1)} = S_{N(n)} + \sum_{k=1}^{Y_{n+1}} (T_{N(n)+1} + k -1) = Y_{n+1} T_{N(n)+1} + \frac12 Y_{n+1}(Y_{n+1}-1).$$
Conditioning on $\mathcal{F}^*_n$ we have 
$$\E(T_{N(n)+1} | \mathcal{F}^*_n) = 1 + S_{N(n)}/N(n),$$ so that
\begin{eqnarray*}
\E\left(S_{N(n+1)}/N(n+1)|\mathcal{F}^*_n\right) &=& S_{N(n)}/N(n) + \E\left( \frac{Y_{n+1}(Y_{n+1}+1)}{N(n)+Y_{n+1}}\big| N(n)\right)\\
&=& S_{N(n)}/N(n) + \sum_{r=1}^{\infty} \frac{q p^{r-1} r(r+1)}{2(N(n)+r)}.
\end{eqnarray*}
The result now follows. 
\end{proof}
 
\subsection{History-dependent branching processes}

More generally, an analysis similar to the above can be carried through for any Markov chain with stationary transition probabilities that suffers reversions to earlier states as determined by a reverting clock. For definiteness, we briefly consider a reverting Galton-Watson branching process.  Let $Z_{k,n}$ be the number of offspring of the $k$th individual in the $n$th generation.  As usual, we assume that these variables are independent and identically distributed and we denote their p.g.f.\ by $W(s)$.  In the reverting case, we suppose that $X_n$, the population size at time $n$ is given by
$$ X_{n+1} = \sum_{k=1}^{X_{U(n)}} Z_{k,n}, \quad n \geq 1,$$
with $X_1 =1$, say, and where $U(n)$ is uniformly distributed on $\{1,\dots,n\}$. Here, and elsewhere, the empty sum is taken to be zero.  Denoting the p.g.f.\ of $X_n$ by $H_n(s)$, we then have
\begin{equation} \label{H}
H_{n+1}(s) = \frac{1}{n} \sum_{k=1}^n H_k(W(s)),
\end{equation}
where  Similar recurrences hold for the moments of $X_n$. 

The functional recurrence \eqref{H} does not appear to be tractable, however arguing as in section \ref{sec:2} we can show that the population size of the reverting Galton-Watson process develops in the same way as the ordinary Galton-Watson process but on a time scale determined by the reverting clock, $T_n$. In terms of the p.g.f\ we have
$$H_n(s) = \E H_{T_n} (s),$$
with similar expressions for the moments, probabilities of extinction, and so on. Thus various limit theorems satisfied by the ordinary Galton-Watson process will hold for the reverting version, by Anscombe's theorem.
  
\section{Conclusion}
We have considered reverting random walks $(X_n;n\geq 1)$ having this key defining property:-- at any time $n$ the next value of the process may be obtained by reverting to a value $X_U$ selected randomly from the history $(X_k;1\leq k\leq n)$ and then requiring the walk to take a step from the value $X_U$.  Such a walk does not have the Markov property.  However, when the step distributions are time homogeneous, it is shown that the reverting process is in fact a version of the usual random walk observed at a random operational time, called the reverting clock.  That is to say, the reverting walk is subordinate to (or directed by) the reverting clock process.  

It is shown that the clock process obeys a central limit theorem, and that its centred time integral yields a martingale which converges to a non-degenerate limit.  These results for the reverting clock are then used to yield the properties of the directed process $(X_n)$.  In particular, for example, we obtain an exact expression for the distribution of the reverting simple random walk and we also show that the reverting walk inherits certain limiting properties that are enjoyed by the untrammelled walk, such as a central limit theorem.

Various variants and generalizations are introduced, including walks that revert non-uniformly to the past and occasionally reverting walks whose reversions occur at the instants of a renewal process.  In particular, we consider the case in which the reversion instants constitute a Bernoulli process, having geometrically distributed intervals between successive reversions.  In principle, a similar analysis can be carried out for any suitable Markov chain with stationary transition probabilities.

Finally, we remark that the nature of many quite straightforward properties of the these reverting processes such as first passage distributions, remain open problems.

\bibliographystyle{apalike}

\vspace{5ex}
\hrule
\vspace{1ex}
\urlstyle{sf}
Email address: {\sf peter.clifford@jesus.ox.ac.uk}\\
URL: \url{https://www.stats.ox.ac.uk/~clifford}

Email address: {\sf david.stirzaker@sjc.ox.ac.uk}\\
URL: \url{https://www.sjc.ox.ac.uk/discover/people/professor-david-stirzaker}

\vspace{2ex}
\hrule
\end{document}